\documentclass[12pt]{amsart}
\usepackage[left=25mm,right=25mm, top=25mm, bottom=25mm]{geometry}
\usepackage{latexsym,array,delarray,amsthm,amssymb,epsfig,setspace,tikz,amsmath,enumerate,mathrsfs,graphicx, mathtools,cite}
\usepackage{geometry}

\usepackage{amssymb,amsmath,amsthm}
\usepackage{tikz}

\usepackage{tabularx}
\usepackage[utf8]{inputenc}
\usepackage[T1]{fontenc}
\usetikzlibrary{patterns}
\usepackage{pgfplots}

\usepackage{hyperref}
\newtheorem{definition}{Definition}

\newtheorem{theorem}{Theorem}

\newtheorem{corollary}{Corollary}

\newtheorem{proposition}{Proposition}
\newtheorem{remark}{Remark}
\newtheorem{question}{Question}
\newtheorem{Conjecture}{Conjecture}
\newtheorem{Fact}{Fact}

\theoremstyle{plain}

\newcommand{\Aff}{\mbox{Aff}}
\newcommand{\Homeo}{\mbox{Homeo}}
\newcommand{\Diff}{\mbox{Diff}}
\newcommand{\ind}{\mbox{$\perp \kern-5.5pt \perp$}}

\pgfplotsset{compat=1.18}


\thanks{}

\date{\today}

\begin{document}

\title{Dynamics of primitive elements under group actions}
\author{Pratyush Mishra\texorpdfstring{}{}}

\address{Department of Mathematics
\\Wake Forest University\\}

\email{mishrapratyushkumar@gmail.com; mishrap@wfu.edu}

\begin{abstract}
    We investigate group actions in which certain primitive elements fix a point, while not all group elements possess this property when acting upon some space. Using similar dynamical tools, we introduce the notion of Nielsen girth and prove the existence of groups with infinite girth but having finite Nielsen girth. 
\end{abstract}

\maketitle
\tableofcontents
\section{Introduction and statement of main results}
This paper provides a computational lens to examine the different facets of dynamics in studying various group actions. Some computations in the proof of Proposition \ref{prop:GL2R} are performed utilizing the computational software Wolfram Mathematica. The Mathematica notebook codes are included in Appendix.
\medskip\

In recent decades, there has been some substantial work on studying primitive elements to analyze the structure of a given group, sometimes under some strong assumption \cite{Platonov}. For instance, it turns out that, by an unpublished result of Zelmanov (see \cite{AB}), there exists finitely generated infinite subgroup $G=\langle s_1,s_2,\dots s_k\rangle \subseteq GL(r,\mathbb{C})$ such that, for all primitive element $g\in G$ , $g^p=1$ for some fixed positive integer $p$, although one would except such a group $G$ to be finite. On the other hand, Shpilrain (see \cite{VS}) studied primitive elements to obtain interesting results such as, endomorphisms of the free group $\mathbb{F}_2$ which act like an automorphism on one specific orbit, are indeed automorphisms of $\mathbb{F}_2$. The author also introduced the notion of generalized primitive elements and obtain results for distinguishing automorphisms among arbitrary endomorphisms of $\mathbb{F}_2$. Moreover, it is well known that if all the elements of a finitely generated subgroup $G\subset GL(n,\mathbb{C})$ are unipotent then $G$ is unipotent and it is conjugate to a subgroup of $n\times n$ upper triangular matrices. One of our major motivations for this paper is the question of Platonov (see Conjecture \ref{thm:platonovconjecture}) asking if all the primitive elements of $G$ are unipotent, can we say the group $G$ is unipotent?
\medskip

For $S=\{s_1,\dots, s_m\}$ an ordered generating set of a group $G$, one can apply the following \textit{Schreier transformations (or Schreier moves)\footnote{Note that the notion of Schreier moves here may differ from many other authors, for instance, Schreier transformations are often called Nielsen transformations.}} to obtain a new generating set $S'$ of $G$

\begin{itemize}
    \item Switching $s_i$ with $s_j$ for some $i\neq j$,
    $$S=\{s_1,s_2,\dots,s_i,\dots,s_j,\dots, s_m\} \mapsto \{s_1,s_2,\dots, s_j,\dots, s_i, \dots, s_m\}=S'.$$
   
    \item Replacing $s_i$ with $s_i^{-1}$ for some $i\in\{1,2,\dots m\}$.
    \item Replacing $s_i$ with $s_is_j$ for some $i\neq j$ and $i,j\in \{1,2,\dots, m\}.$ 
\end{itemize}  

We denote $S\sim^r_\mathcal{S} S',$ if the generating set $S'$ can be obtained from $S$ by applying atmost $r$ many Schreier transformations. Similarly, we denote $( S\sim_{\mathcal{S}} S')$ if $S'$ can be obtained from $S$ by applying finitely many Schreier transformations.

\medskip

\begin{definition}\label{defn:primitive}
Fix $m\geq d(G),$ where $d(G)$ denotes the minimal cardinality of the generating set of $G$. Let $G=\langle S_0\rangle,$ be a finitely generated group, where $S_0=\{ g_1,g_2,\dots,g_m\}$ is a fixed generating set. We define inductively the subsets $\mathcal{R}_n, n\geq 0$ of $G$ as $$\mathcal{R}_0=\{S_0\}, \mathcal{R}_{n+1}=\{ S \hspace{0.1cm}|\hspace{0.1cm}\langle S \rangle =G, |S|=m, \exists S'\in \mathcal{R}_n\hspace{0.1cm}\text{such that}\hspace{0.1cm} S\sim^1_{\mathcal{S}} S'\}$$ 
where $S\sim_{\mathcal{S}}^1S'$ means $S'$ is obtained from $S$ by applying atmost one Schreier transformation. An element $x\in G$ is primitive\footnote{Our notion of primitive element depends on choice of the $S_0$.} if there exists $S\in \displaystyle \mathop{\bigcup}_{n\geq 0}\mathcal{R}_n$ such that $x\in S$.
\end{definition}

\begin{definition}
    We define the $n^{\text{th}}$ step primitive elements of $G$ as $$\mathcal{P}_n=\{x\in G\hspace{0.1cm}|\hspace{0.1cm} x\hspace{0.1cm}\text{is primitive such that}\hspace{0.1cm} \exists S\in \mathcal{R}_n \hspace{0.1cm}\text{with}\hspace{0.1cm} x\in S \}$$
\end{definition} 

\noindent With the above definition, the collection of all the primitive elements of $G$ is given by
$$ \mathcal{P}=\displaystyle \mathop{\bigcup}_{n\geq 0}\mathcal{P}_n.$$

\noindent Also, note that

$$\mathcal{P}_0\subseteq \mathcal{P}_1\subseteq \mathcal{P}_2 \subseteq \dots \subseteq \mathcal{P}_n \subseteq \dots$$ 

\begin{Fact}\label{fact:primitivedefinition}
\cite{Whitehead} For free groups $\mathbb{F}_n=\langle a_1,a_2,\dots,a_n\rangle, n\geq 2$, an element $g\in \mathbb{F}_n$ is primitive if and only if $g=\phi(a_i)$ for some automorphism $\phi: \mathbb{F}_n\to \mathbb{F}_n$ and some element $a_i$ for $1\leq i\leq n$ in the given set of generators $\{a_1,a_2,\dots, a_n\}$ of $\mathbb{F}_n$.
\end{Fact}

Here is a classical folklore conjecture which is a more general version of some of the related conjectures studied by Platonov and Potapchik \cite{Platonov}:

\begin{Conjecture}\label{con:primitiveelementnilpotent}
If all primitive elements of a finitely generated subgroup $G\subseteq GL(n,\mathbb{C})$, are unipotent, then is the group $G$ necessarily nilpotent?
\end{Conjecture}

For the group $G=\langle S \rangle$ as in Definition \ref{defn:primitive}, consider the action of $G$ on some manifold $X$, for $g\in G$, define $$Fix(g)=\{ x\in X \hspace{0.1cm}|\hspace{0.1cm} g\cdot x=x  \}.$$

All of our questions \ref{thm:holder}, \ref{thm:Q2}, and \ref{thm:firstlevelgenerators} listed below are motivated by the question mentioned in Conjecture \ref{con:primitiveelementnilpotent}.

\begin{question}\label{thm:holder}
     If $\mathcal{P}$ acts freely (without fixed points) on some space $X$, then is it true that $G$ is abelian?
\end{question}

Question \ref{thm:holder} can be seen as a generalization of the classical H\"{o}lder's theorem (Proposition 2.2.29, \cite{AN}) to general group actions. H\"{o}lder's theorem states that if a group $\Gamma$  acts freely by orientation preservation homeomorphisms on the real line then $\Gamma$ is necessarily Abelian. Recently, in \cite{AC},  Akhmedov and Cohen introduced the notion of H\"{o}lder manifold, that is, an orientable manifold $X$ with boundary $\partial X$ is H\"{o}lder if any subgroup $\Gamma \leq Homeo_+(X,\partial X)$ acting freely on $X$ and fixing the boundary $\partial X$ pointwise, is Abelian. For instance, $S^1$ and interval $[0,1]$ are H\"{o}lder. Moreover, even dimensional rational homology spheres are H\"{o}lder, which follows immediately from Lefschetz fixed-point theorem. Furthermore, Question \ref{thm:holder} seems to be related to another question posed by Navas, Calegari, and Rolfsen in \cite{CR} and \cite{AN1} about the left orderability of $Homeo(I^n,\partial I^n)$ for $n\geq 2$. In a more recent work, Hyde \cite{H} disproved Calegari's and Rolfsen's question by constructing a counterexample and proving that $Homeo(I^2,\partial I^2)$ is not left orderable. In Proposition \ref{prop:primitivefree}, we construct an example of a group $\Gamma\leq Homeo_+(\mathbb{R})$ where all the primitive elements of $\Gamma$ act freely on $\mathbb{R}$ but the group $\Gamma$ does not act freely.

\begin{question}\label{thm:Q2}
    If the Fix$(s)\neq \emptyset$, for all $s\in \mathcal{P}$, then is it true that Fix$(g)\neq \emptyset$, for all $g\in G$?
\end{question}

\begin{question}\label{thm:firstlevelgenerators}
    What if $\mathcal{P}$ is replaced by $\mathcal{P}_1$ in Question \ref{thm:Q2}?
\end{question}

Interestingly, Question \ref{thm:Q2} is related to the classical Lie-Kolchin theorem, which was proved by S. Lie in 1876 in the context of Lie algebras and then for linear algebraic groups by  E. Kolchin in 1948. 
\medskip

\textbf{Lie-Kolchin Theorem:} \textit{If $\Gamma$ is a connected solvable linear algebraic group over an algebraically closed field, then for any linear action of $G$ on a non-zero finite dimensional vector space $V$ by linear transformations, (which is also the same as viewing it as a linear representation $\rho: G \to GL(V)$), there exists a one-dimensional subspace $W\subset V$ such that for all $g(W)=W$ for all $g\in G$.}
\medskip

\noindent Before investigating Questions \ref{thm:holder},\ref{thm:Q2}, and \ref{thm:firstlevelgenerators} for various group actions, we will take an algebraic digression and address certain group-theoretic questions which also seem to be of dynamical flavor. 
\medskip

\begin{remark}
   (\cite{MKS}) Note that Fact \ref{fact:primitivedefinition} follows from another well-known fact that the group of outer automorphisms of $\mathbb{F}_n, n\geq 2$ (the quotient group $\frac{Aut(\mathbb{F}_n)}{Inn(\mathbb{F}_n)}$) is generated by the Schreier moves, where $Inn(\mathbb{F}_n)$ means the group of inner automorphisms of $\mathbb{F}_n$.
\end{remark}

In connection to Conjecture \ref{con:primitiveelementnilpotent}, the following classical result is known about the linear groups:
\begin{Fact}
(\cite{JH}) If all the elements of a finitely generated linear group $G$ are unipotent, then the group $G$ is nilpotent.
\end{Fact}

In \cite{Platonov}, Platonov and Potapchik constructed the following $4\times 4$ matrices:
\begin{equation}\label{equation:PlatonovPotapchik}
  x=\begin{bmatrix}
    1 & 0 & 0 & 0\\
    0 & 1 & 0 & 0\\
    1 & 0 & 1 & 0\\
    0 & -1 & 0 & 1
\end{bmatrix}, y=\begin{bmatrix}
    1 & 1 & 0 & 0\\
    0 & 1 & 1 & 0\\
    0 & 0 & 1 & 1\\
    0 & 0 & 0 & 1
\end{bmatrix}.  
\end{equation}

\noindent The authors \cite{Platonov} noticed that all elements $x^ny$ for $n\in \mathbb{Z}$ are unipotent, but the group generated by $x$ and $y$ is not unipotent. Moreover, in the same paper \cite{Platonov}, the authors proved the following result:

\begin{theorem}\label{thm:pp} (\cite{Platonov})
For any non-abelian free group $\mathbb{F}_n$, $n\geq 3$, let $$\phi: Aut(\mathbb{F}_n)\to GL(m,\mathbb{C}),$$ be a finite-dimensional representation of $Aut(\mathbb{F}_n)$. Assume that all the images $\phi(x)$, where $x$ is a primitive element of $Inn(\mathbb{F}_n)$ are unipotent and the number of Jordan blocks in $\phi(x)$ does not exceed $n$, then the image $\phi(Inn(\mathbb{F}_n))$ is unipotent.    
\end{theorem}

Theorem \ref{thm:pp} partially answers the following general conjecture posed by Platonov and Potapchik in the same paper \cite{Platonov}:

\begin{Conjecture}\label{thm:platonovconjecture}
For a given group $G=\langle S\rangle$, with a fixed finite generating set $S$, if all the primitive elements of $G$ are unipotent, then is the whole group $G$ generated by $S$ necessarily unipotent?  
\end{Conjecture}

\subsection{Girth of groups}
The girth of a finite graph is defined as the length of the shortest non-trivial cycle in it. The girth of graphs is studied in combinatorics, number theory, and graph theory. For any given positive integer $n$ and $d\geq 2$, Erd\"{o}s and Sachs proved the existence of $(d, n)$-graphs. Finding the smallest such graph is still an active area of research. In a similar spirit, Schleimer \cite{S} introduced the notion of the girth of groups. For a finitely generated group $G=\langle S | R\rangle$, the girth of the group $G$ denoted $girth(G)$, is defined as $$girth(G)=\sup_{S\subset G, \langle S\rangle=G}\bigg \{girth(Cay(G,S))\hspace{0.1cm}\bigg |\hspace{0.1cm} |S|<\infty \bigg \},$$ where $girth(Cay(G,S))$ is the girth of the Cayley graph of the group $G$ with respect to finite generating set $S$.  A substantial amount of work on girth of groups was done by Akhmedov \cite{azer1}, \cite{azer2} including the notion of Girth Alternative and he proved it for certain classes of finitely generated groups including linear groups not isomorphic to $\mathbb{Z}$, word hyperbolic groups, one relator groups, etc. Later on, Yamagata and Nakamura studied the Girth Alternative for subgroups of of \mbox{Out}$(F_n)$ \cite{Y} and for mapping class groups \cite{N} respectively. Recently, Akhmedov and the author of this paper introduced the girth alternative for a subclass of HNN extensions and amalgamated free products of finitely generated groups and also produced counterexamples to show that the alternative fails in the general class of HNN extensions (and for amalgamated free products respectively) \cite{AM}.
\medskip

Motivated by Conjecture \ref{thm:platonovconjecture}, in this paper we consider the following notion of `Schreier girth' which we learned from A. Akhmedov (Definition \ref{definition: Schreiergirth}). In a similar spirit, we extend the concept of Schreier girth to `Nielsen girth' (in Section \ref{section:NielsenGirth}). In Proposition \ref{prop:finiteNgirth}, we produced a class of groups with finite Nielsen girth but having infinite girth.

\begin{definition}\label{definition: Schreiergirth}
    We define the $k$-step Schreier girth of a group $G$, denoted $\mathcal{S}_k$-girth$(G)$ as follows $$\mathcal{S}_k\text{-girth}(G)=\inf_{\langle S \rangle=G,|S|=k}\bigg \{\sup_{\langle S' \rangle=G, S\sim_{\mathcal{S}} S'}\hspace{0.1in}\text{girth (Cay}(G,S'))\bigg \}$$ where $S\sim_{\mathcal{S}} S'$ means the generating set $S'$ can be obtained from $S$ by applying finitely many Schreier transformations and for $k\geq d(G)$, where $d(G)$ is the cardinality of a minimal generating set of $G$.
\end{definition}
Before stating the main results of this paper, we first need to address why the Questions \ref{thm:holder}-\ref{thm:firstlevelgenerators} makes sense. To understand this, let $V$ be a finite-dimensional vector space over an arbitrary field $\mathbb{K}$ of dimension $n$. A chain $V_0\subseteq V_1\subseteq \dots V_{i-1}\subseteq V_i$  of subspaces is called a flag if $V_0 \subsetneq V_1\subsetneq \dots \subsetneq V_i$, i.e., each subspace is proper. This chain is called a complete flag if $i = n$ and dim $V_i = i, 0\leq i\leq n$ (thus $V_0 = 0$ and $V_n = V$). The set of all flags of $V$ will be denoted as $\mathcal{F}(V)$. $\mathcal{F}(V)$ has a structure of an algebraic variety over $\mathbb{K}$. One can put more structures on  $\mathcal{F}(V)$. For example, if $\mathbb{K}$ is a local field (such as $\mathbb{R}, \mathbb{C}$, and $\mathbb{Q}_p$), then the flag variety  $\mathcal{F}(V)$ can be a given a structure of a topological space. Let $\mathcal{F}_c(V)$ be the subspace of $\mathcal{F}(V)$ consisting of complete flags. 
 \medskip

Now, if a group $G$ acts on $V$ by linear transformations, then it also acts on $\mathcal{F}(V)$ as well as on $\mathcal{F}_c(V)$. Also, $g\in G$ has a fixed point in $\mathcal{F}_c(V)$ if and only if $g$ is a matrix representation in the upper triangular form. We can still ask Questions \ref{thm:holder}-\ref{thm:firstlevelgenerators} for this action. Let us note that if $\mathbb{K}$ is algebraically closed, then by Jordan Form Theorem, every element fixes a point (thus Questions \ref{thm:holder}-\ref{thm:firstlevelgenerators} are trivially resolved). However, these questions are still interesting if  $\mathbb{K}$ is not algebraically closed. In the case of $\mathbb{R}$, we discuss this question below in Section \ref{sec:GLactsRP1}.
\medskip

For every complete flag $\alpha $, we define $W(\alpha ) = \displaystyle \mathop{\oplus}_{0\leq i\leq n-1}V_{i+1}/V_i$ where $\alpha $ is represented by the chain $V_0\subseteq V_1\subseteq \dots V_{n-1}\subseteq V_n$, and let $\Omega (V) = \{(\alpha , W(\alpha )) : \alpha \in  \mathcal{F}_c(V)\}$. Then $G$ acts naturally on $\Omega (V) $, and an element $g\in G$ fixes a point if and only if it is unipotent. 
\medskip

The question studied by Platonov and Potapchik (Conjecture \ref{thm:platonovconjecture}) can now be formulated as Question \ref{thm:Q2} for the action of $G$ on $\Omega (V)$ or even for general group actions.  
\medskip

\subsection{Statement of main results} Here are the statements of the main results in this paper:

\begin{proposition}\label{prop:finiteNgirth}
For $n\geq 2$, the group $\Gamma_n$ defined as 
\begin{equation}
       \Gamma_n=  \langle a,b|[a,b]^n=1 \rangle,
    \end{equation}
has finite $\mathcal{N}_2{\text{-}girth}$ but $girth(\Gamma_n)=\infty $.  
\end{proposition} 

\begin{proposition}\label{prop: AffR}
    There exist finitely generated subgroups $\langle f,g \rangle \leq \mbox{Aff}_+(\mathbb{R})$ such that $Fix(s)\neq \emptyset$ for all first step primitive elements, that is $s\in \mathcal{P}_1$ but $Fix(h)=\emptyset$ for some $h \in \langle f,g\rangle$.  
\end{proposition}

\begin{proposition}\label{prop:abelian}
    For any finitely generated subgroup $G\leq SL(2,\mathbb{R})$, if all the first step primitive elements of $G$ fix a point in $\mathbb{H}^2$, then $G$ fixes a point in $\mathbb{H}^2$ and hence $G$ is abelian.
\end{proposition}

\begin{proposition}\label{prop:GL2R}
    Consider the action of $GL(2,\mathbb{R})$ on $\mathbb{R}P^1$ by usual matrix multiplication. Then, there exists a finitely generated subgroup $\Gamma \leq GL(2,\mathbb{R})$ such that $Fix(s)\neq \emptyset$ for all $s\in \mathcal{P}_1$ but $Fix(g)=\emptyset$ for some $g\in \Gamma$.   
\end{proposition}

Let $f\in \mbox{PL}_+(I)$, where $\mbox{PL}_+(I)$ denote the group of orientation-preserving piecewise linear homeomorphisms of the interval $[0,1]$. The $C_0$-norm in $\mbox{PL}_+(I)$ is defined as $$||f||=\max_{x\in [0,1]}|f(x)-x|.$$A subgroup $G\leq \mbox{PL}_+(I)$ is $C_0$-dense if $G$ is dense in the $C_0$ topology in $\mbox{PL}_+(I)$. For example, Thompson's group $F$ is a $C_0$-dense subgroup of $\mbox{PL}_+(I)$.

\begin{proposition}\label{prop:PLI}
 There exists a finitely generated subgroup $\Gamma\leq \mbox{PL}_+(I),$ such that all the first step primitive elements fix a point in $(0,1)$, but there is a group element $g\in \Gamma$ such that $Fix(g)=\emptyset$.
\end{proposition}


\begin{proposition}\label{prop:primitivefree}
    There exists a subgroup $\Gamma\leq \mbox{Homeo}_+(\mathbb{R})$ such that all the primitive elements of $\Gamma$ act freely but the entire group doesn't act freely. 
\end{proposition}
In Section \ref{section:NielsenGirth}, we develop the study of Nielsen girth and Schreier girth. Moreover, we produce a class of groups with finite Nielsen girth but having infinite girth in Propositions \ref{prop:finiteNgirth}. In Proposition \ref{prop: AffR}, \ref{prop:abelian}, \ref{prop:GL2R}, and \ref{prop:PLI}, we investigate Questions \ref{thm:holder}-\ref{thm:firstlevelgenerators} for various group actions.

\section{Nielsen Girth}\label{section:NielsenGirth}
For a given generating set $S=\{s_1,\dots , s_n\}$ of a group $G$, one can apply the following \textit{Nielsen moves} to obtain a new generating set $S'$ of $G$ as follows:
\begin{itemize}
    \item Replacing $s_k$ with $s_j^{\pm}s_k$ for some $k\neq j$ and $j,k\in \{1,2,\dots, n\}$ i.e., $$S=\{s_1,s_2,\dots,s_k,\dots, s_n\} \sim_N \{s_1,s_2,\dots, s_j^{\pm}s_k,\dots,s_n\}=S'.$$
    \item Replacing $s_k$ with $s_ks_j^{\pm}$ i.e., $$S=\{s_1,s_2,\dots,s_k,\dots, s_n\} \sim_N \{s_1,s_2,\dots, s_ks_j^{\pm},\dots,s_n\}=S'.$$
\end{itemize} 

We write $S' \sim_N S$ to denote that $S'$ is obtained from $S$ by applying finitely many Nielsen move. The switching of any two elements from the generating set is not allowed in Nielsen moves. Motivated by the notion of Schreier girth of a group defined in the earlier section, here we define the Nielsen girth as follows
\begin{definition}
    The $k$-step Nielsen girth of a group $G$, denoted $\mathcal{N}_k$-girth$(G)$ is defined as, $$\mathcal{N}_k\text{-girth}(G)=\inf_{\langle S \rangle=G,|S|=k}\bigg \{\sup_{\langle S' \rangle=G, S\sim_N S'}\hspace{0.1in}\text{girth (Cay}(G,S'))\bigg \},$$ $k\geq d(G)$, where $d(G)$ is the cardinality of the minimal generating set of $G$.
\end{definition}

\begin{remark}
    A key difference between Nielsen girth and Schreier girth is that if $S \sim_N S',$ then for any distinct $x,y\in S$, and $x',y'\in S'$, the commutator $[x', y']$ is either the same as $[x,y]$ or it is conjugate to $[x,y]$ by some element in $G$. 
\end{remark}

Here are some immediate corollaries:
\begin{corollary}\label{cor:2}
    Given a finitely generated group $G=\langle S \rangle$, the following implications holds $$\mathcal{N}_k\text{-girth}(G)=\infty \implies \mathcal{S}_k\text{-girth}(G)=\infty \implies girth(G)=\infty.$$ for every $k\geq d(G),$ where $d(G)$ is the cardinality of the minimal generating set $S$ of $G$.
\end{corollary}
\begin{proof}
 As every Nielsen move is a Schreier move, it follows from the definition that 
 \[\mathcal{N}_k\mbox{-}girth(G) =\infty \implies \mathcal{S}_k\mbox{-}girth(G)=\infty. \] Now for the other implication, if $girth(G)<\infty$ then for every sequence of finite set of generators $\{S_n\}$ of $G$ (in particular, where each $S_n$ is of cardinality $k$ and $S_i \sim_{\mathcal{S}} S_j$, i.e. related by a Schreier move in $\{S_n\})$ for all positive integers $i$ and $j$, there exists a natural number $M$ such that $girth(G,S_n)<M$ which implies $S_k\text{-}girth(G)<\infty.$ 
\end{proof}

\begin{corollary}
    Let $G=\langle S\rangle$ be a finitely generated non-cyclic group satisfying a law, then $\mathcal{N}_k$-girth(G)$<\infty$ and $\mathcal{S}_k$-girth(G)$<\infty$ for $k\geq d(G)$, where $d(G)$ is the cardinality of the minimal generating set of $S$ of $G$.
\end{corollary}
\begin{proof}
    As non-cyclic groups satisfying a law have finite girth \cite{azer1}, by Corollary \ref{cor:2}, it follows that $$\mathcal{N}_k\textit{-girth}(G), \mathcal{S}_k\textit{-girth}(G)<\infty.$$
\end{proof}

\begin{corollary}
    Let $G$ be a finitely generated group with $N\unlhd G$ such that $G/N$ is not cyclic then if for some positive integer $k\geq d(G)$, where $d(G)$ is the cardinality of the minimal generating set of $G$, $$\mathcal{S}_k\text{-}girth(G/N)=\infty\implies girth(G)=\infty.$$
\end{corollary}
\begin{proof}
   It follows from \cite{azer1} that $girth(G)<\infty \implies girth(G/N)<\infty$ and then by Corollary \ref{cor:2}, we have $S_k\text{-}girth(G)<\infty$ for all $k\geq d$.
\end{proof}

\begin{remark}
    However, the converse of Corollary \ref{cor:2} is not true in general, Proposition \ref{prop:finiteNgirth} gives a class of counterexamples.
\end{remark}

\subsection{Proof of Proposition \ref{prop:finiteNgirth}}
\begin{proof}
    We will show that $\mathcal{N}_2\text{-girth}(\Gamma_n)\leq 4n$. Let $S=\{a,b\}$ and $S_r=\{a_r,b_r\}$ be any generating set such that $S_r$ is obtained from $S$ by applying finitely many Nielsen moves i.e., $S=S_1 \sim^1_N S_2 \sim^1_N S_3\sim^1_N \dots \sim^1_N S_r$, where $S_i\sim_N^1 S_{i+1}$ means $S_{i+1}$ is obtained from $S_i$, $1\leq i\leq r-1$ by applying exactly one Nielsen move. Under a Nielsen move $S_k\sim_N S_{k+1}$, we have these eight possibilities: 
    \medskip    
     \begin{itemize}
         \item  $S_k=\{a_k, b_k\}\sim^1_N \{a_kb_k^{\pm},b_k\}=S_{k+1}$,
         \item $S_k=\{a_k,b_k\}\sim^1_N \{b_k^{\pm}a_k,b_k\}=S_{k+1}$,
         \item $S_k=\{a_k,b_k\}\sim^1_N \{a_k,a_k^{\pm}b_k\}=S_{k+1}$, and
         \item $S_k=\{a_k,b_k\}\sim^1_N \{a_k,b_ka_k^{\pm}\}=S_{k+1},$
     \end{itemize}
     \medskip
\noindent     for $1\leq k\leq r$, then the commutator in the $k+1$ step, $S_{k+1}$ is either conjugate or remains the same to the commutator in step $k$, $S_k$ as shown below:
\medskip
     \begin{enumerate}
         \item $[a_kb_k,b_k]=a_kb_kb_k(a_kb_k)^{-1}(b_k)^{-1}=a_kb_kb_kb_k^{-1}a_k^{-1}b_k^{-1}=[a_k,b_k]$,
         \item $[b_ka_k,b_k]=b_ka_kb_k(b_ka_k)^{-1}b_k^{-1}=b_ka_kb_ka_k^{-1}b_k^{-1}b_k^{-1}=b_k[a_k,b_k]b_k^{-1}$,
         \item $[a_kb_k^{-1}, b_k]=[a_k,b_k]$,
         \item $[b_k^{-1}a_k,b_k]=b_k^{-1}[a_k,b_k]b_k$,
         \item $[a_k,a_kb_k]=a_k[a_k,b_k]a_k^{-1}$,
         \item $[a_k,a_k^{-1}b_k]=a_k^{-1}[a_k,b_k]a_k$,
         \item $[a_k,b_ka_k]=[a_k,b_k]$, and
         \item $[a_k,b_ka_k^{-1}]=[a_k,b_k]$.
     \end{enumerate}
     \medskip
     
    However, we know that the order of a group element remains invariant under conjugation, that is ord$([b_k^{\pm}a_k,b_k])$=ord$([a_kb_k^{\pm},b_k])$ =ord$([a_k,a_k^{\pm}b_k])$ =ord$([a_k,b_ka_k^{\pm}])$=ord$([a_k,b_k])$. Therefore, $\mathcal{N}_2\text{-girth}(\Gamma_n)$ is at most $4n$ (i.e. ord$([a_k,b_k]^n)=$ ord$((g[a_k,b_k]g^{-1})^n)$=$4n$ at each step $k$, where $1\leq k\leq r$ and $g\in \{a^{\pm}, b^{\pm}\}$) and hence,  $\mathcal{N}_2\text{-girth}(\Gamma_n)<\infty$. 
    \medskip    
    But, as Tits Alternative holds for one-relator group asserts (\cite{KS}) and from the classification of virtually solvable subgroups of one relator group, it follows that any 1-relator group is of the following type:  1) $\mathbb{Z}/n\mathbb{Z}$- a finite cyclic group of order $n$, 2) $\mathbb{Z}$-an infinite cyclic group of integers, 3) $\mathbb{Z}^2$- the group of the integral grid in $\mathbb{R}^2$, 4) Klein group $\mathbb{K}=\langle a,b| b^{-1}ab=a^{-1}\rangle$, and 5) the solvable Baumslag-Solitar groups $BS(1,n)=\langle a,b| b^{-1}ab=a^n\rangle$ for any integer $|n|\geq 1$ (or $BS(m,1)$ respectively for $|m|\geq 1$). Therefore, $\Gamma_n$ for any integer $n\geq 2$ is not virtually solvable as it does not fall under any of the above categories of virtually solvable subgroups of 1-relator groups. Hence, $\Gamma_n$ has infinite girth by Girth Alternative for 1-relator groups (\cite{azer2},\cite{KS}).
\end{proof}
   
\begin{definition}
    We define the Nielsen girth ($\mathcal{N}\text{-girth}$) and Schreier girth ($\mathcal{S}\text{-girth}$) of $G$ as $$\mathcal{N}\text{-girth}(G)=\inf_{\langle S \rangle=G} \{\sup_{\langle S'\rangle=G, S'\sim_{N} S}\hspace{0.1in}\text{girth Cay}(G,S')\} \hspace{0.1cm}\text{and}$$ $$\mathcal{S}\text{-girth}(G)=\inf_{\langle S \rangle=G} \{\sup_{\langle S'\rangle=G, S'\sim_{\mathcal{S}} S}\hspace{0.1in}\text{girth Cay}(G,S')\},$$ where $S$ is the minimal generating set of $G$.
\end{definition}

\begin{Conjecture}\label{conj:3}
    Any finitely generated linear group is either virtually solvable or has an infinite $\mathcal{S}$-girth.
\end{Conjecture}
Conjecture \ref{conj:3} is precisely the \textbf{Schreier Girth Alternative} for linear groups, which is still an open problem and one of our major motivations.

\section{Action of the affine group on the real line}

In this section, we study Question \ref{thm:holder} and Question \ref{thm:Q2} for the group of orientation-preserving affine homeomorphisms of the real line $$\textit{Aff}_+(\mathbb{R})=\bigg \{f \in \Homeo_+(\mathbb{R})\hspace{0.1cm}\bigg |\hspace{0.1cm} f(x)=ax+b, a>0, b\in \mathbb{R}\bigg\}\simeq \Bigg \{\begin{bmatrix}
    a & b\\
    0 & 1
\end{bmatrix}, a>0, b\in \mathbb{R} \Bigg \}.$$ As per the above matrix representation of the group $\Aff_+(\mathbb{R})$, the action of $\Aff_+(R)$ on the real line is defined by M\"{o}bious transformation, that is for any $f=\begin{bmatrix}
    a & b\\
    0 & 1
\end{bmatrix}\in \Aff_+(\mathbb{R}),$ and $x\in\mathbb{R}$
$$f\cdot x=\begin{bmatrix}
    a & b \\
    0 & 1
\end{bmatrix}\cdot x=\frac{ax+b}{0x+1}=ax+b.$$
To answer Question \ref{thm:holder}, note that a matrix in $\Aff_+(\mathbb{R})$ acts freely on $\mathbb{R}$ if and only if $b\neq 0$ and $a=1$, which follows from the fixed-point condition. Therefore, any two matrices of the form $A=\begin{bmatrix}
    1 & b\\
    0 & 1
\end{bmatrix}, B=\begin{bmatrix}
    1 & c\\
    0 & 1
\end{bmatrix}$ acts freely on $\mathbb{R}$ and commute. This also agrees with the well-known H\"{o}lder theorem (Proposition 2.2.29, \cite{AN}). And we answer positively Question \ref{thm:holder}.

\subsection{Proof of Proposition \ref{prop: AffR}}
\begin{proof}
We will now answer Question \ref{thm:Q2} negatively for the action of $\Aff_+(\mathbb{R})$ on $\mathbb{R}$. To find the fixed point condition, let $f=\begin{bmatrix}
    a & b\\
    0 & 1
\end{bmatrix}$ and $g=\begin{bmatrix}
    c & d\\
    0 & 1
\end{bmatrix}$ be two non-trivial elements in $\Aff_+(\mathbb{R})$, we have for any non-zero integer $n$ \begin{equation}
    f^ng=\begin{bmatrix}
        a^n & b(\frac{1-a^n}{1-a})\\
        0 & 1
    \end{bmatrix}\begin{bmatrix}
        c & d\\
        0 & 1
    \end{bmatrix}=\begin{bmatrix}
        a^nc & a^nd+b(\frac{1-a^n}{1-a})\\
        0 & 1
    \end{bmatrix},
\end{equation}
\begin{equation}
    f^ng^{-1}=\begin{bmatrix}
        a^n & b(\frac{1-a^n}{1-a})\\
        0 & 1
    \end{bmatrix}\begin{bmatrix}
        \frac{1}{c} & -\frac{d}{c}\\
        0 & 1
    \end{bmatrix}=\begin{bmatrix}
        \frac{a^n}{c} & -\frac{a^nd}{c}+b(\frac{1-a^n}{1-a})\\
        0 & 1
    \end{bmatrix},
\end{equation}

\begin{equation}
    g^nf=\begin{bmatrix}
        c^n & d(\frac{1-c^n}{1-c})\\
        0 & 1
    \end{bmatrix}\begin{bmatrix}
        a & b\\
        0 & 1
    \end{bmatrix}=\begin{bmatrix}
        c^na & c^nb+d(\frac{1-c^n}{1-c})\\
        0 & 1
    \end{bmatrix},
\end{equation}

\begin{equation}
    g^nf^{-1}=\begin{bmatrix}
        c^n & d(\frac{1-c^n}{1-c})\\
        0 & 1
    \end{bmatrix}\begin{bmatrix}
        \frac{1}{a} & -\frac{b}{a}\\
        0 & 1
    \end{bmatrix}=\begin{bmatrix}
        \frac{c^n}{a} & -\frac{c^nb}{a}+d(\frac{1-c^n}{1-c})\\
        0 & 1
    \end{bmatrix}.
\end{equation} From all the above equations, the condition for the first step primitive elements $f^ng^{\pm},g^nf^{\pm}$ to have fixed points is given by
\begin{equation}\label{equation:fixedpointcondition}
    c^na, \frac{c^n}{a}, \frac{a^n}{c}, a^nc \neq 1,
\end{equation}
Now, let $f$ and $g$ be two elements of $\Aff_+(\mathbb{R})$ satisfying the fixed-point condition obtained in Equation \ref{equation:fixedpointcondition}. The commutator $[f,g]$ does not satisfy the fixed-point condition Equation $\ref{equation:fixedpointcondition}$, provided $-d-cb+b+ad\neq 0$
\begin{equation} \label{eq: fix}
[f,g]=fgf^{-1}g^{-1}=\begin{bmatrix}
    1 & -d-cb+b+ad\\
    0 & 1
    \end{bmatrix}.
\end{equation} Hence, the group $\langle f, g \rangle$ generated by elements $f$ and $g$ (satisfying Equation \ref{equation:fixedpointcondition} and $-d-cb+b+ad\neq 0$) does not have a fixed-point, answering negatively Question \ref{thm:Q2}. Here is a concrete choice of $a,b,c$ and $d$ satisfying Equations \ref{equation:fixedpointcondition} and \ref{eq: fix} i.e.,

$$f=\begin{bmatrix}
    2 & 3\\
    0 & 1
\end{bmatrix}, g=\begin{bmatrix}
    3 & 4\\
    0 & 1
\end{bmatrix}, [f,g]=\begin{bmatrix}
    1 & -2\\
    0 & 1
\end{bmatrix}.$$
Hence we completed the prove of the Proposition \ref{prop: AffR}.
\end{proof}

\section{Modular group action on the hyperbolic plane}

In this section, we consider the M\"{o}bius action of the group $SL(2,\mathbb{R})$ on the hyperbolic plane $\mathbb{H}^2$. Here we consider the upper half space model of the space $\mathbb{H}^2=\{z\in \mathbb{C}| Im(z)>0\}$ and the boundary of $\mathbb{H}^2$, $\partial \mathbb{H}^2=\mathbb{R}$. The action of $SL(2,\mathbb{R})$ on $\mathbb{H}^2$ is by M\"{o}bius transformations defined as $$\begin{bmatrix}
    a & b\\
    c & d
\end{bmatrix}\cdot z=\frac{az+b}{cz+d}\hspace{0.1in}\text{where}\hspace{0.1in}Im(z)>0, z\in \mathbb{C}.$$

\noindent The discriminant for a matrix $A=\begin{bmatrix}
    a & b\\
    c & d
\end{bmatrix}\in SL(2,\mathbb{R})$ under the M\"{o}bius action is given by $$D=(a+d)^2-4= |Trace(A)|^2-4,$$ as $ad-bc=1$ for $A$ being in $SL(2,\mathbb{R})$. 
\medskip

Based on the fixed points, the $SL(2,\mathbb{R})$ action on $\mathbb{H}^2$ are classified as follows:
\medskip
\begin{enumerate}
    \item \textit{Elliptic}, if $D<0;$ then $|Trace(A)|<2$. Any such matrix has a unique fixed point in $\mathbb{H}^2$.
    \item \textit{Parabolic}, if $D=0;$ then $|Trace(A)|=2$. Any such matrix has a fixed point in the boundary $\partial \mathbb{H}^2=\mathbb{R}$.
    \item \textit{Hyperbolic}, if $D>0;$ then $|Trace(A)|>2$. Any such matrix has two distinct fixed points in the boundary $\partial \mathbb{H}^2=\mathbb{R}$.
\end{enumerate}
\medskip
Now, we are ready to prove Proposition \ref{prop:abelian}.

\subsection{Proof of Proposition \ref{prop:abelian}}
\begin{proof}
We assume $G$ is generated by two non-identity matrices $A=\begin{bmatrix}
    a & b\\
    c & d
\end{bmatrix}$ and $B=\begin{bmatrix}
    p & q\\
    r & s
\end{bmatrix}$ in $SL(2,\mathbb{R})$. Our proof will generalize to finitely generated groups whose generating set contains more than two elements. Let $z_0, z_1$ be two distinct points in $\mathbb{H}^2$ such that
\begin{equation}\label{equation:HypFix}
    A\cdot z_0=z_0, B\cdot z_1=z_1.
\end{equation}
As we know, given any two distinct points in $\mathbb{H}^2$, there exists a unique geodesic (vertical line or a half circle) in $\mathbb{H}^2$ passing through them. Then, we can find matrices in $SL(2,\mathbb{R})$ that map $z_0, z_1$ to $i, ri$ respectively in the imaginary axis for some $r>0$. So, without loss of generality, we can assume $z_0=i$ and $z_1=ri$ for some $r>0$ in Equation (\ref{equation:HypFix}), which implies 
$$A\cdot i=\begin{bmatrix}
    a & b\\
    c & d
\end{bmatrix}\cdot i=i\implies a=-c,b=d.$$

Similarly, 
$$B\cdot i=\begin{bmatrix}
    a' & b' \\
    c' & d'
\end{bmatrix}\cdot (ri)=ri\implies 
c'=-\lambda b',$$ where $\lambda=-\frac{1}{r^2}, r\neq 0$ and as $A,B \in SL(2,\mathbb{R})$, the matrices $A$ and $B$ are precisely of this form $$A=\begin{bmatrix}
    \cos{\theta} & \sin{\theta}\\
    -\sin{\theta} & \cos{\theta}
\end{bmatrix}, B=\begin{bmatrix}
    \cos{\eta} & \sin{\eta}\\
    -\lambda \sin{\eta} & \cos{\eta}
\end{bmatrix}, \lambda=r^2>0$$
where $\theta$ and $\eta$ can not be a multiple of $2n\pi$ for $n\in \mathbb{Z}$.  Now, without loss of generality we replace $B$ by $B'$ a conjugate of $B$ by a diagonal matrix with $\lambda, \lambda^{-1}$ on the diagonal, as shown below, $$B'=\begin{bmatrix}
    \lambda & 0\\
    0 & \lambda^{-1}
\end{bmatrix}\begin{bmatrix}
    \cos{\eta} & \sin{\eta}\\
    -\sin{\eta} & \cos{\eta}
\end{bmatrix}\begin{bmatrix}
    \lambda^{-1} & 0\\
    0 & \lambda
\end{bmatrix}.$$ Using trigonometric identities, we have$$A^n=\begin{bmatrix}
    \cos{n\theta} & \sin{n\theta}\\
    -\sin{n\theta} & \cos{n\theta}
\end{bmatrix}, {B'}^k=\begin{bmatrix}
    \lambda & 0\\
    0 & \lambda^{-1}
\end{bmatrix}\begin{bmatrix}
    \cos{k\eta} & \sin{k\eta}\\
    -\sin{k\eta} & \cos{k\eta}
\end{bmatrix}\begin{bmatrix}
\lambda^{-1} & 0\\
0 & \lambda
\end{bmatrix}.$$

The trace of $A^nB'^k$ is given by the following expression:
\begin{multline}\label{equa:expression}
   Trace(A^n{B'}^k)=2\cos{n\theta} \cos{k\eta}-(\lambda^2+\lambda^{-2})\sin{n\theta}\sin{k\eta}\\=2\cos(n\theta -k\eta)-(2+\lambda^2+\lambda^{-2})\sin{n\theta}\sin{k\eta}
    =2\cos(n\theta-k\eta)-(\lambda+\lambda^{-1})^2\sin{n\theta}\sin{k\eta}.
\end{multline}
\noindent Assume without loss of generality, $0<\theta, \eta <\pi$, then we have the following two cases:
\medskip

\textbf{Case 1:} If $\theta, \eta$ are rational, we can choose $k=1$ and some $n$ arbitrary such that the expression obtained in Equation \ref{equa:expression}, $2\cos(n\theta-\eta)-(\lambda+\lambda^{-1})^2\sin{n\theta}\sin{\eta}$ is less than $-2$ and hence we are done.
\medskip

\textbf{Case 2:} If one of $\theta$ or $\eta$ is irrational, without loss of generality, assume $\theta$ is irrational. Choose $k=1$, so we have $\sin{k\eta}=\sin{\eta}\neq 0$. Then, for $\epsilon >0$, we can choose integer $n$ such that $\pi+2\pi r\leq (n\theta -\eta)\leq \pi+2\pi r+\epsilon$ for some integer $r$ and $\sin{n\theta}\sin{\eta}>0$. Then, for sufficiently small $\epsilon$, we have $$2\cos(n\theta-\eta)-(\lambda+\lambda^{-1})^2\sin{n \theta}\sin{k \eta}<-2.$$ 

\noindent Hence, $A^nB'$ does not have a fixed point unless $G$ there is a fixed point in $G$.
\end{proof}
In the proof of Proposition \ref{prop:abelian}, we also showed the following result,
\begin{corollary}
    In a non-abelian subgroup $G\leq SL(2,\mathbb{R})$, a freely acting element in the group $G$ always exists.
\end{corollary}

\section{General linear group action on the real projective line}\label{sec:GLactsRP1}

$GL(2,\mathbb{R})$ naturally acts on $\mathbb{R}^2-\{(0,0)\}$ by usual matrix multiplication, and this action lifts to an action of $GL(2,\mathbb{R})$ on the real projective line $\mathbb{R}P^1$. For a matrix $A\in GL(2,\mathbb{R})$ to have a fixed point in $\mathbb{R}P^1$ we must have $$\begin{bmatrix}
    a & b\\
    c & d
\end{bmatrix}\cdot \begin{bmatrix}
    x\\
    y
\end{bmatrix}=\lambda\begin{bmatrix}
    x\\
    y
\end{bmatrix}
\implies \begin{bmatrix}
    a- \lambda & b\\
    c & d-\lambda
\end{bmatrix}\cdot \begin{bmatrix}
    x\\
    y
\end{bmatrix}=\begin{bmatrix}
    0\\
    0
\end{bmatrix}.$$

\noindent For the above equation to have real roots, we must have
\begin{equation}\label{equation:condtn}
    (a+d)^2-4(ad-bc)\geq 0 \implies |Trace(A)|\geq 2|Det(A)|.
\end{equation}
Hence, the fixed point condition for the action of $GL(2,\mathbb{R})$ on $\mathbb{R}P^1$ is $|Trace(A)|\geq 2|Det(A)|$.

\subsection{Proof of Proposition \ref{prop:GL2R}}
\begin{proof}
Consider the subgroup $\Gamma=\langle A,B \rangle \leq GL(2,\mathbb{R})$ which do satisfy the conditions in Equation \ref{equation:condtn} $$A=\begin{bmatrix}
    0.13 & 0.95\\
    0 & 1.9
\end{bmatrix}, B=\begin{bmatrix}
    0.15 & 0\\
    -0.06 & 1.9
\end{bmatrix}.$$ 
Note that $Det(A)=0.247, Det(B)=0.285, Trace(A)=2.03,$ and $Trace(B)=2.05$. We will show that Equation \ref{equation:condtn} holds for the first step primitive elements, $A^mB$ and $B^mA$ (similar reasoning can be given for $A^mB^{-1}$ and $B^mA^{-1}$).

\begin{multline}\label{equation:AmB}
    Trace(A^mB)=0.15 (0.13^m) - 0.06\{(-0.536723)0.13^m + (0.536723)1.9^m\} + 
 1.9^{m+1}\\ 
 =(0.18220338)(0.13)^m+(1.9)^m(1.86779662).
\end{multline}
 \begin{multline}\label{equa:AmBdet}
     2Det(A^mB)=\\2[(-0.061186)(0.247)^m + (0.061186)1.9^{2m} + 
   (0.18220)(0.13)^m(1.9)^{1 + m} - (0.032203)(1.9)^{1+ 2m}]\\
   =(0.569998)(0.247)^m.
 \end{multline}
 \begin{multline}\label{equation:BmA}
     Trace(B^mA)= 0.13 (0.15^m) + 
 0.95\{(0.0342857) (0.15)^m - (0.0342857)(1.9)^m\} + 1.9^{m+1}\\
 = (0.162571415)(0.15)^m +1.867428585(1.9)^m.
 \end{multline}
\noindent\begin{equation}\label{equa:BmA}
2Det(B^mA)=2(0.13) (0.15)^m (1.9)^{1+m}=0.494(0.15)^m(1.9)^m=0.494(0.285)^m.
\end{equation}

Now, here we give a ping-pong argument, in the $Trace(A^mB)$ Equation \ref{equation:AmB}, for positive integer values of $m$, the second term $(1.9)^m$ dominates and $|Trace(A^mB)|\geq 2|Det(A^mB)|$ holds. For negative integer values of $m$ in Equation \ref{equation:AmB}, the first term $(0.13)^m$ dominates, and still $|Trace(A^mB)|\geq 2|Det(A^mB)|$ holds. By similar arguments, $|Trace(B^mA)|\geq 2|Det(B^mA)|$. Similar arguments can be made for the other first-level primitive elements $A^mB^{-1}$ and $B^mA^{-1}$. Moreover, it turns out that even elements, namely $A^nB^m$ and $B^mA^n$ for all $m,n\in \mathbb{Z}$ satisfy the condition of Equation $\ref{equation:condtn}$. However, the commutator $[B,A]$ has a trace of less than two i.e., $$[B,A]=B^{-1}A^{-1}BA=\begin{bmatrix}
    1.2 & -83.7949\\
    0.0357341 & -1.66194
\end{bmatrix}.$$ Hence, as $|Trace([B,A])|=0.461943<2.00001=2|Det[B,A]|$, the commutator $[B,A]$ fail to satisfy the condition of Equation \ref{equation:condtn} and therefore $\langle A,B \rangle$ does not fixes a point in $\mathbb{R}P^1$ answering negatively Question \ref{thm:firstlevelgenerators}.
    
\end{proof}

\begin{remark}
    Similar constructions as in Proposition \ref{prop:GL2R} can also be arranged in the higher dimensions for $GL(n+1, \mathbb{R})$ action on $\mathbb{R}P^{n}$ for all positive integers $n\geq 1$.
\end{remark}

\section{Action of the group of the piecewise linear homeomorphims of interval}


Consider the group generated by two elements $f$ and $g$ in $\mbox{PL}_+(I)$ shown in Figure \ref{fig:11}, we will prove that $\Gamma=\langle f,g \rangle$ answer negatively to Question \ref{thm:firstlevelgenerators}, proving Proposition \ref{prop:PLI}.
\subsection{Proof of Proposition \ref{prop:PLI}}
\begin{proof}
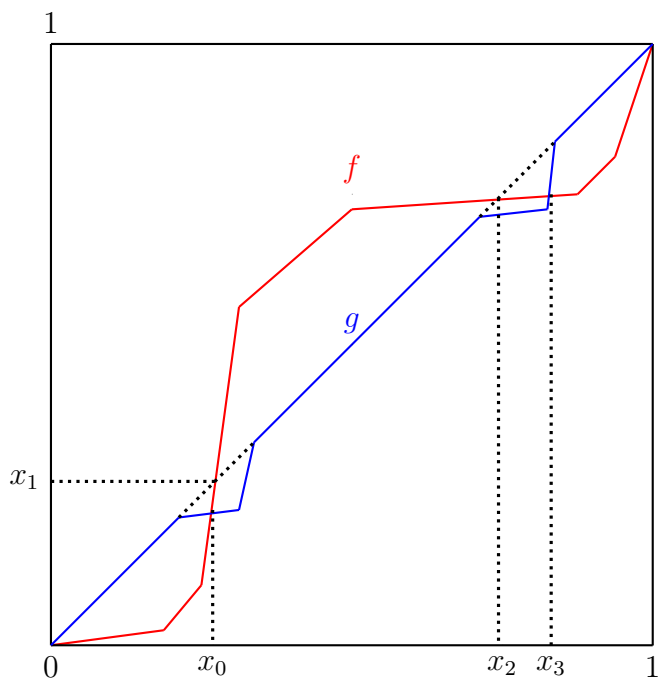
\begin{figure}
    \centering
   
\begin{tikzpicture}[thick,scale=1, every node/.style={scale=1}]
    
    \draw [color=red](0,0)--(1.5,0.2);
    \draw [color=red](1.5,0.2)--(2,0.8);
    \draw [color=red](2,0.8)--(2.5,4.5);
    \draw [color=red] (2.5,4.5)--(4,5.8);
    \draw [color=red] (4,5.8)--(7,6);
    \draw [color=red] (7,6)--(7.5,6.5);
    \draw [color=red] (7.5,6.5)--(8,8);
    \draw (0,0)--(0,8);
    \draw (0,0)--(8,0);
    \draw (8,0)--(8,8);
    \draw (0,8)--(8,8);
    \draw [color=blue](0,0)--(1.7,1.7);
    \draw [color=blue](1.7,1.7)--(2.5,1.8);
    \draw [color=blue](2.5,1.8)--(2.7, 2.7);
    \draw [color=blue](2.7,2.7)--(5.7,5.7);
    \draw [color=blue](5.7,5.7)--(6.6,5.8);
    \draw [color=blue](6.6,5.8)--(6.7,6.7);
    \draw [color=blue](6.7,6.7)--(8,8);
    \draw[very thick, dotted] (2.15,1.8)--(2.15,0);
    \draw[fill](2.15,0) circle [radius=0] node [below] [color=black] {$x_0$};
    \draw[fill](6.65,0) circle [radius=0] node [below] [color=black] {$x_3$};
    \draw[very thick, dotted] (6.65,6)--(6.65,0);
    \draw[very thick, dotted] (1.7,1.7)--(2.7,2.7);
    \draw[very thick, dotted] (5.7,5.7)--(6.7,6.7);
    \draw[very thick, dotted] (2.18,2.18)--(0, 2.18);
    \draw[very thick, dotted] (5.95,5.95)--(5.95,0); 
    \draw[fill](0, 2.2) circle [radius=0] node [left] [color=black] {$x_1$};
    \draw[fill](6,0) circle [radius=0] node [below] [color=black] {$x_2$};
    
    \draw[fill](4,4) circle [radius=0] node [above] [color=blue] {$g$};
    \draw[fill](4,6) circle [radius=0] node [above] [color=red] {$f$};
    \draw[fill](0,8) circle [radius=0] node [above] {1};
    \draw[fill](8,0) circle [radius=0] node [below] {1};
    \draw[fill](0,0) circle [radius=0] node [below] {0};
\end{tikzpicture}

\caption{The group $\langle f,g \rangle$ answer negatively Question \ref{thm:firstlevelgenerators} in $\mbox{PL}_+(I)$}
    \label{fig:11}
\end{figure}

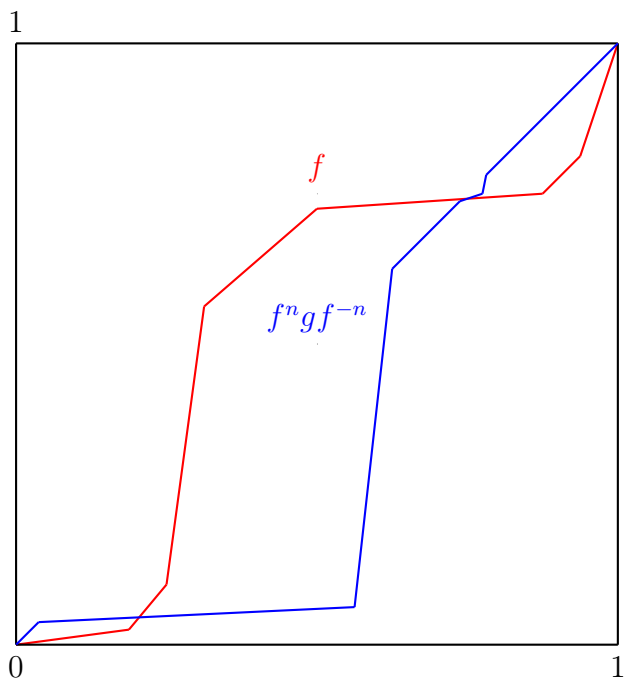
\begin{figure}
    \centering

\begin{tikzpicture}[thick,scale=1, every node/.style={scale=1}]

    \draw [color=red](0,0)--(1.5,0.2);
    \draw [color=red](1.5,0.2)--(2,0.8);
    \draw [color=red](2,0.8)--(2.5,4.5);
    \draw [color=red] (2.5,4.5)--(4,5.8);
    \draw [color=red] (4,5.8)--(7,6);
    \draw [color=red] (7,6)--(7.5,6.5);
    \draw [color=red] (7.5,6.5)--(8,8);
    \draw (0,0)--(0,8);
    \draw (0,0)--(8,0);
    \draw (8,0)--(8,8);
    \draw (0,8)--(8,8);
    \draw [color=blue](0,0)--(0.3,0.3);
    \draw [color=blue](0.3,0.3)--(4.5,0.5);
    \draw [color=blue](4.5,0.5)--(5, 5);
    \draw [color=blue](5,5)--(5.9,5.9);
    \draw [color=blue](5.9,5.9)--(6.2,6);
    \draw [color=blue](6.2,6)--(6.25,6.25);
    \draw [color=blue](6.25,6.25)--(8,8); 
    \draw[fill](4,4) circle [radius=0] node [above] [color=blue] {$f^ngf^{-n}$};
    \draw[fill](4,6) circle [radius=0] node [above] [color=red] {$f$};
    \draw[fill](0,8) circle [radius=0] node [above] {1};
    \draw[fill](8,0) circle [radius=0] node [below] {1};
    \draw[fill](0,0) circle [radius=0] node [below] {0};
\end{tikzpicture}

 \caption{Existence of $g_1=f^{n}gf^{-n}$ in $\langle f,g\rangle$ for $n>>1$}
    \label{fig:12}
\end{figure}

\begin{figure}
    \centering

\begin{tikzpicture}[thick,scale=1, every node/.style={scale=1}]
   
    \draw [color=red](0,0)--(1.5,0.2);
    \draw [color=red](1.5,0.2)--(2,0.8);
    \draw [color=red](2,0.8)--(2.5,4.5);
    \draw [color=red] (2.5,4.5)--(4,5.8);
    \draw [color=red] (4,5.8)--(7,6);
    \draw [color=red] (7,6)--(7.5,6.5);
    \draw [color=red] (7.5,6.5)--(8,8);
    \draw (0,0)--(0,8);
    \draw (0,0)--(8,0);
    \draw (8,0)--(8,8);
    \draw (0,8)--(8,8);
    \draw [color=blue](0,0)--(2,2);
    \draw [color=blue](2,2)--(2.3,2.1);
    \draw [color=blue](2.3,2.1)--(2.4, 2.4);
    \draw [color=blue](2.4,2.4)--(2.7,2.7);
    \draw [color=blue](2.7,2.7)--(7.6,3.3);
    \draw [color=blue](7.6,3.3)--(7.8,7.8);
    \draw [color=blue](7.8,7.8)--(8,8);
    \draw[fill](4,2) circle [radius=0] node [above] [color=blue] {$f^{-n}gf^{n}$};
    \draw[fill](4,6) circle [radius=0] node [above] [color=red] {$f$};
    \draw[fill](0,8) circle [radius=0] node [above] {1};
    \draw[fill](8,0) circle [radius=0] node [below] {1};
    \draw[fill](0,0) circle [radius=0] node [below] {0};
\end{tikzpicture}

 \caption{Existence of $g_2=f^{-n}gf^n$ in $\langle f,g \rangle$ for $n>>1$}
   \label{fig:13}
\end{figure}
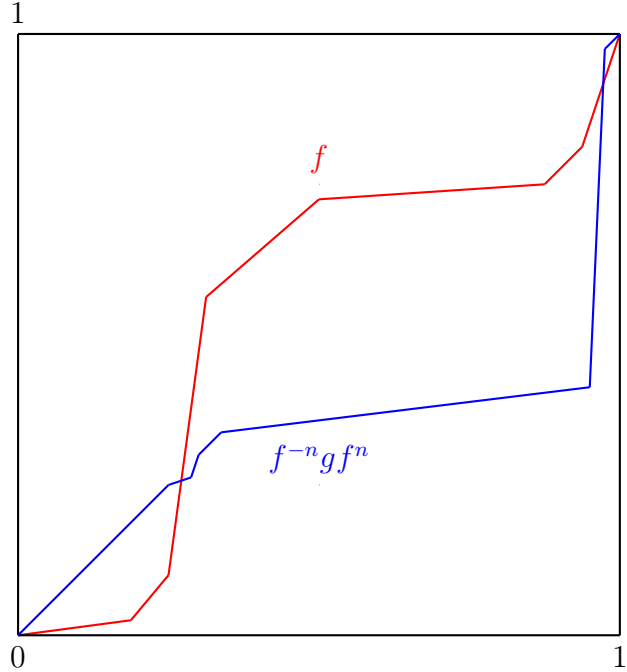
\medskip

Let $f(x_0)=g(x_0), f(x_1)=x_1, g(x_2)=x_2$ and $f(x_3)=g(x_3)$ for $0<x_0<x_1<x_2<x_3<1$ as shown in the Figure \ref{fig:11}.
\medskip

\textbf{Claim 1:} For any non-zero integer $n,l$, $f^n(x')=g^l(x')$ for some $x'\in (0,1)$.
\medskip

Let $l=1$, and we will show that for any integer $n$, there exists $x_n'\in (0,1)$ such that $f^n(x_{n}')=g(x_{n}')$, meaning $g^{-1}f^n$ always fixes a point in $(0,1)$. A similar argument works for any integer $l>1$. Assume $n>0$ ( a similar argument applies for $n<0$) and we will show that $F(x)=f^n(x)-g(x)$ is zero for some $x\in (0,1)$. Note that
\begin{equation}
    F(x_0)=f^n(x_0)-g(x_0)<f^n(x_0)-f(x_0)<0,
\end{equation}
as $f(x_0)<x_0, f(x_0)=g(x_0)$, further as $f$ is orientation preserving, $f^n(x_0)<f(x_0)$ for any integer $n>1$. But\begin{equation}
    F(x_1)=f^n(x_1)-g(x_1)=x_1-g(x_1)>0,
\end{equation}
as $g(x_1)<x_1$. So, by the intermediate value theorem, there always exists $x_n'\in (0,1)$ such that $f^n(x_n')=g(x_n')$. Similar reasoning can also be made for the $l>1$ case, and we completed the proof of Claim 1. 
\medskip

Claim 1 guarantees that all first-order primitive elements fix a point in $(0,1)$. Then letting $g_1=f^{n}gf^{-n}$ and $g_2=f^{-n}gf^{n}$ for $n>>1$ we get Figure \ref{fig:12} and \ref{fig:13} respectively (for more details on such increasing bump techniques see \cite{B}). 

\textbf{Claim 2:} $(g_1g_2)^kf$ acts freely on $(0,1)$ for all $k\geq 1, k\in \mathbb{Z}$.
\medskip

For small $\epsilon >0$, let $x\in (0,\epsilon)$ (similar reasoning works for other cases). Then, $$g_1g_2f(x)<g_1(x)=x.$$
So, $g_1g_2f$ does not fix a point in $(0,\epsilon)$. Hence the group $\langle f,g\rangle$ contains an element that does not fix a point in $(0,1)$ answering negatively Question \ref{thm:firstlevelgenerators}.
\end{proof}


\begin{remark}
Similar constructions as in Proposition \ref{prop:PLI} can also be arranged in the groups $\Homeo_+(I)$ and $\Diff_+^{r}(I)$ for any $r$, $1\leq r\leq \infty$ by smoothing the breakpoints in our example of $f$ and $g$ in Figure \ref{fig:11}.
\end{remark}
We conjecture that the choice of $f$ and $g$ as in Proposition \ref{prop:PLI} can be made in any $C_0$-dense subgroup of $\mbox{PL}_+(I)$:
\begin{Conjecture}
 For any finitely generated $C_0$-dense subgroup $\Gamma$ in $\mbox{PL}_+(I)$, there exists a generating set $S$ of $\Gamma$ such that the first step primitive elements (with respect to $S$) in $\Gamma$ fix a point in $I$, but there is a group element $g\in \Gamma$ such that $Fix(g)=\emptyset$.
\end{Conjecture}

\section{Remarks in connection to Conjecture  \ref{con:primitiveelementnilpotent}, \ref{thm:platonovconjecture}}\label{thm:remarksinconjectures}
\subsection{The action of \texorpdfstring{$GL(n, \mathbb{C})$}{} on \texorpdfstring{$\mathbb{C}^n$}{} by linear transformations}
It is easy (trivial) to find two matrices $A,B \in GL(n,\mathbb{C}), n\geq 2$ with a common eigenvector of eigenvalue 1, i.e. $Fix(A)\cap Fix(B)\neq \emptyset$, such that the subgroup generated by $A$ and $B$ fixes the common eigenvector answering positively Question \ref{thm:Q2} and Question \ref{thm:firstlevelgenerators} in the group $\langle A,B \rangle$.

\subsection{The action of the Heisenberg group on \texorpdfstring{$\mathbb{R}^3$}{}}
The group of all $3\times 3$ upper triangular matrices with entries 1 on the main diagonal is the Heisenberg group $\mathbf{H}$,
\begin{center}
   \[ \Bigg \{\begin{bmatrix}
        1 & a & c\\
        0 & 1 & b\\
        0 & 0 & 1
    \end{bmatrix}
 \Bigg | a,b,c \in\mathbb{R}.  \Bigg \}\]
 \end{center}
 Let us consider any two matrices $A,B\in \mathbf{H}$, then all first step primitive elements $A^nB$ always fix the vector $(1,0,0)$, and hence the whole group fixes the vector $(1,0,0)$ in $\langle A,B\rangle$. What is interesting is the fact that $\mathbf{H}$ has many faithful representations in $GL(n,\mathbb{C})$ for $n\geq 3$ but none of them provide a counterexample to Conjecture \ref{thm:platonovconjecture}.

\section{Obstruction to properly discontinuous actions}
Properly discontinuous actions are group actions studied by topologists and geometric group theorists. They possess many desirable properties and play a significant role in understanding the structure of spaces and groups. For instance, for any covering map $p:E\to X$, the action of the deck transformation group on the $E$ is properly discontinuous\footnote{Note that the definition of properly discontinuous actions varies from context to context}.

\begin{definition}
An action of a group $G$ on a topological space $X$ is properly discontinuous if, for all $x\in X$, there exists a neighborhood $U$ of $x$ such that all non-trivial element of $g\in G$ moves $U$ i.e., $$g\neq 1\implies gU\cap U=\emptyset, \forall g\in G.$$
\end{definition}

\begin{definition}
    An action of a group $G$ on a topological space $X$ is free if any $x\in X$ is moved by all non-trivial element $g\in G$ i.e., $$g\neq 1\implies g.x\neq x, \forall g\in G.$$
\end{definition}

\begin{remark}
    Properly discontinuous actions are free. To see this, suppose to the contrary, assume the action is properly discontinuous but not free, then there exists a non-trivial $g\in G$ with $g.x=x$ for some $x\in X$. For any neighborhood $U$ of $x$, $g.U\cap U\neq \emptyset$ (as it always contains $x$), and therefore, the action is not properly discontinuous.
\end{remark}

\begin{remark}\label{thm:obstruction}
    Note that if $G$ acts $X$ and if the action admits first step primitive elements in $G$, meaning $g^nh^{\pm}, h^{n}g^{\pm}$ fixes some point under the action of $G$ on $X$, then such an action is not free and hence not properly discontinuous. 
\end{remark}

Remark \ref{thm:obstruction} says that the first step primitive elements are obstructions to properly discontinuous group actions, meaning to study Question \ref{thm:Q2} more, we need to look for more non-properly discontinuous actions, which we think is also interesting in its own way. What is more interesting is the question, what if the first step primitive elements are not obstructions? In some instances, this already implies that the action is free. We would like to raise the following questions:

\begin{question}\label{thm:firstlevelfreeprimitive}
   If all the first-step primitive elements $\mathcal{P}_1$ act freely then can we say that the entire group acts freely?
\end{question}

\begin{question}\label{thm:primitivefree}
    If all the primitive elements $\mathcal{P}$ act freely, then can we say that the entire group acts freely?
\end{question}

\subsection{Proof of Proposition \ref{prop:primitivefree}}
\begin{proof}
Let $(G,\preceq)$ be a countable left orderable group (i.e. $G$ admits a left translation invariant total order, $x\preceq y \implies gx\preceq gy$ for any $x,y,g\in G$) such that $\mathbb{Z}^2$ acts on $G$. We choose a left ordering on $(\mathbb{Z}^2,\leq_\mathbb{R})$, by realizing it as a subgroup $\langle 1,\sqrt{2} \rangle$ of $(\mathbb{R},+)$. Then, we make $\mathbb{Z}^2\ltimes G=\{((n,m);g)|(n,m)\in \mathbb{Z}, g\in G\}$ a left orderable group by lexicographic ordering, i.e. $((n_1,m_1);g)\preceq ((n_2,m_2);h)$ if $(n_1,m_1)\leq_\mathbb{R} (n_2,m_2)$.

\begin{definition}
    Let $(\Gamma,\preceq)$ be a left orderable group. An element $x\succ 0$ in $\Gamma$ is dominant, if for any arbitrary $y\in \Gamma,$ there exists $n\in \mathbb{Z}$ such that $x^n \succ y$ and $x^{-n}\prec y$.
\end{definition}

A countable left orderable group acts faithfully by orientation preserving homeomorphisms on the real line (Theorem 2.2.19, \cite{AN}), therefore $\mathbb{Z}^2\ltimes G\leq \mbox{Homeo}_+(\mathbb{R})$. Moreover, using this representation it follows that dominant elements of $\mathbb{Z}^2\ltimes G\leq \mbox{Homeo}_+(\mathbb{R})$ act freely on the real line.
\medskip

Let $\phi,\psi \in \mathbb{Z}^2\ltimes G$ such that $\phi=((1,0);f)$ and $\psi=((0,1);g)$ for some non-trivial $f,g\in G$ such that $[\phi,\psi]\neq e$. Clearly, all primitive elements in the subgroup $\langle \phi,\psi \rangle$ are dominant (if a primitive element is not positive then one can take it's inverse) and hence act freely on $\mathbb{R}$. However, the commutator $[\phi,\psi]$ has the $(0,0)$ entry in the $\mathbb{Z}^2$ coordinate and the subgroup $\langle \phi,\psi \rangle$ does not act freely on $\mathbb{R}$ (as if it acts freely, then by the H\"{o}lder's theorem (Proposition 2.2.29, \cite{AN}), $\mathbb{Z}^2\ltimes G$ has to be abelian which is a contradiction), answering negatively Question \ref{thm:primitivefree} and \ref{thm:firstlevelfreeprimitive}. Interestingly, this proposition also answers negatively Question \ref{thm:holder}.
\end{proof}

\section{Question related to Hyde's counterexample}{\label{thm:Hydesrelatedquestion}}

Another strong reformulation of Question \ref{thm:Q2} is restricting to a given fixed set of generators (as generators of a group are trivially primitive, we are essentially considering a smaller subset of the collection of all the primitive elements $\mathcal{P}$).

\begin{question}\label{thm:Q6}
For a given group $G=\langle S \rangle$ acting on a manifold $X$ with a fixed generating set $S$, if $Fix(s_i)\neq \emptyset$ for all $s_i\in S$, then is it true that $Fix(G)\neq \emptyset$, meaning the group $G$ has a global fixed point?
\end{question}

Interestingly, our Question \ref{thm:Q6} is related to the Hyde's recent counterexample of non-left orderability of $\mbox{Homeo}_+(D,\partial D)$, \cite{H}. Some examples studied by Triestino \cite{Triestino} are $\mathbb{Z}^2$ and the Klein bottle group acting on the real line by orientation preserving homeomorphisms has the property satisfying Question \ref{thm:Q6} positively. To be more precise, Lemma 3 in \cite{Triestino} states that any faithful action of $\mathbb{Z}^2=\langle f,g |f^{-1}gf=g\rangle$ or the Klein bottle group $K=\langle f,g | f^{-1}gf=g^{-1} \rangle$ on the real line, if the generators $f$ and $g$ have fixed points then the whole group $\mathbb{Z}^2$ and $K$ respectively has a global fixed point. 
\medskip

Here is an adaptation of the proof of Lemma 3 in \cite{Triestino} to a broader class of groups that we present in Proposition \ref{prop:ZKgeneralization} which also answers positively Question \ref{thm:Q6}. In the proof of Proposition \ref{prop:ZKgeneralization}, we will use the following fact which follows immediately from the definition: for any $f,g\in \mbox{Homeo}_+(\mathbb{R})$, $Fix(fgf^{-1})=f(Fix(g))$ and $Fix(g^n)=Fix(g)$ for all non-zero integer $n\in \mathbb{Z}$.

\begin{proposition}\label{prop:ZKgeneralization}
Consider the group $$\Omega_k=\bigg \langle f, g_1, \dots, g_k\hspace{0.1cm}\bigg|\hspace{0.1cm} fg_1^{n_1}f^{-1}=g_1^{m_1}, fg_2^{n_2}f^{-1}=g_2^{m_2}, \dots, fg_k^{n_k}f^{-1}=g_k^{m_k} \bigg \rangle$$ where $m_i,n_j\in \mathbb{Z}\backslash{\{0\}}$ for all $1\leq i,j<\infty$. For any faithful action of $\Omega_k$ on $\mathbf{R}$ by homeomorphisms without a global fixed point, if Fix$(g_i)\neq \emptyset, 1\leq i < \infty$ and $\displaystyle \mathop{\bigcap}_{i\geq 1} Fix(g_i)\neq \emptyset$, then Fix$(f)=\emptyset$.
\end{proposition}

\begin{proof}
Assume to the contrary Fix$(f)\neq \emptyset$. Let $x\in \displaystyle \mathop{\bigcap}_{i\geq 1} Fix(g_i)$. Then, the following sequence $\{f^n(x)\}_{n\in \mathbb{Z}}$ accumulates around some point $p\in Fix(f)$ and 
\begin{equation}\label{eq:18}
    \{f^n(x)\}_{n\in \mathbb{Z}}\subset \displaystyle \mathop{\bigcap}_{i\geq 1} Fix(g_i)
\end{equation}
Equation \ref{eq:18} follows from the fact that $$f(Fix(g_i))=f(Fix(g_i^{n_i}))=Fix(fg_i^{n_i}f^{-1})=Fix(g_i^{m_i})=Fix(g_i)$$ for all $1\leq i< \infty$. As $\displaystyle \mathop{\bigcap}_{i\geq 1} Fix(g_i)$ is a closed set (being an intersection of closed set), we have $p\in \displaystyle \mathop{\bigcap}_{i\geq 1} Fix(g_i)$ which is a contradiction to the fact that $p$ is fixed by the whole group $\Gamma$ (as $p$ is fixed by all the generators).
\end{proof}

Note that the group $\Omega_k$ is a large group in the sense that it contains a copy of non-abelian free group generated by $g_1,g_2,\dots,g_k$. 
\medskip


Although, we suspect, there are non-trivial group actions where Questions  \ref{thm:holder}, \ref{thm:firstlevelgenerators}, \ref{thm:primitivefree} holds positively, we think such group actions are rare. Our investigation for various group actions in this paper was driven by searching for such positive results, where the Conjecture \ref{thm:platonovconjecture} of Platonov and Potapchik holds to be true. For example, we suspect that in higher dimensions $(n\geq 5),$ there exist classes of subgroups of $GL(n,
\mathbb{C})$ where the property of all the primitive elements being unipotent would coincide with the entire subgroup being unipotent. And this is quite interesting because such examples would help us understand better the dynamics of primitive elements in obtaining deep results in connection to Conjecture \ref{thm:platonovconjecture}. 

\section{Acknowledgements.}
We are grateful to Azer Akhmedov for very useful conversations on the manuscript and for introducing to us the notion of Schreier girth and for pointing out the example in Proposition \ref{prop:primitivefree}. We are thankful to Fan Yang for valuable comments on the initial draft of this paper.

\section{Appendix}
\begin{enumerate}

\item Mathematica Code for $GL(2,\mathbb{R})$ action on $\mathbb{R}P^1$:

$$\left(A=\left(
\begin{array}{cc}
 0.13 & 0.95 \\
 0 & 1.9 \\
\end{array}
\right);\right) \left(B=\left(
\begin{array}{cc}
 0.15 & 0 \\
 -0.06 & 1.9 \\
\end{array}
\right);\right)$$

\noindent\text{MatrixPower}[A,m].\text{MatrixPower}[B,1]\\
\noindent\text{MatrixPower}[B,m].\text{MatrixPower}[A,1]\\
\noindent\text{MatrixPower}[A,m].\text{MatrixPower}[B,-1]\\
\noindent\text{MatrixPower}[B,m].\text{MatrixPower}[A,-1]\\
\noindent\text{Tr}[\text{MatrixPower}[A,m].\text{MatrixPower}[B,1]]\\
\noindent\text{Tr}[\text{MatrixPower}[B,m].\text{MatrixPower}[A,1]]\\
\noindent\text{Det}[\text{MatrixPower}[A,m].\text{MatrixPower}[B,1]]\\
\noindent\text{Det}[\text{MatrixPower}[B,m].\text{MatrixPower}[A,1]]\\
\noindent\text{MatrixPower}[B,-1].\text{MatrixPower}[A,-1].\text{MatrixPower}[B,1].\text{MatrixPower}[A,1]\\
\text{Tr}[\text{MatrixPower}[B,-1].\text{MatrixPower}[A,-1].\text{MatrixPower}[B,1].\text{MatrixPower}[A,1]]
\end{enumerate}

\end{document}